\documentclass[11pt]{amsart}
\usepackage[top=30truemm,bottom=30truemm,left=30truemm,right=30truemm]{geometry} 
\usepackage{amsmath, amssymb, array}
\usepackage[all]{xy}
\usepackage{ascmac}
\usepackage[dvips]{graphicx}
\usepackage{color}
\usepackage{amscd}
\usepackage[driverfallback=dvipdfm]{hyperref}
\usepackage{amsrefs}
\usepackage{cleveref}
\usepackage{shuffle}
\usepackage{mathrsfs}
\usepackage{listings}
\usepackage{comment}

\newfont{\cyr}{wncyr10 scaled\magstep0}

\newcommand{\setmid}{\mathrel{}\middle|\mathrel{}} 
\newcommand{\op}[1]{\mathop{\mathrm{#1}}}

\newcommand{\C}{\mathbb{C}}

\newcommand{\Q}{\mathbb{Q}}
\newcommand{\R}{\mathbb{R}}

\newcommand{\Z}{\mathbb{Z}}

\DeclareMathOperator{\Gal}{Gal} %Galois group
\DeclareMathOperator{\GL}{GL} %general linear group
\DeclareMathOperator{\Tr}{Tr} %trace

\theoremstyle{plain}
 \newtheorem{theorem}{Theorem}[section]
 \crefname{theorem}{Theorem}{Theorems}
 \newtheorem{proposition}[theorem]{Proposition}
 \crefname{proposition}{Proposition}{Propositions}
 \newtheorem{lemma}[theorem]{Lemma}
 \crefname{lemma}{Lemma}{Lemmas}
 \newtheorem{corollary}[theorem]{Corollary}
 \crefname{corollary}{Corollary}{Corollaries}
 
 \crefname{conjecture}{Conjecture}{Conjectures}
 
 \crefname{hypothesis}{Hypothesis}{Hypotheses}
 \newtheorem{question}[theorem]{Question}
 \crefname{question}{Question}{Questions}
 
 \crefname{problem}{Problem}{Problems}

\theoremstyle{definition} 
 \newtheorem{definition}[theorem]{Definition}
 \crefname{definition}{Definition}{Definitions}
 \newtheorem{example}[theorem]{Example}
 \crefname{example}{Example}{Examples}
 \newtheorem{remark}[theorem]{Remark}
 \crefname{remark}{Remark}{Remarks}

\title{Galois trace forms of type $A_{n}, D_{n}, E_{n}$ for odd $n$}
\author{Riku Higa}
\address[Riku Higa]{Department of Mathematics \\ Faculty of Science and Technology \\ Tokyo University of Science, 2641, Yamazaki, Noda, Chiba, Japan}
\email{6123702@ed.tus.ac.jp \\ 6121510@ed.tus.ac.jp}
\author{Yoshinosuke Hirakawa}
\address[Yoshinosuke Hirakawa]{Department of Mathematics \\ Faculty of Science and Technology \\ Tokyo University of Science, 2641, Yamazaki, Noda, Chiba, Japan}
\email{hirakawa\_yoshinosuke@rs.tus.ac.jp \\  hirakawa@keio.jp}
	\thanks{This research was supported in part by KAKENHI 21K13779.}
\begin{document}

\maketitle

\begin{abstract}
Let $p$ be an odd prime number
and $\zeta_{p} := \exp(2\pi i/p)$.
Then, it is well-known that
the $A_{p-1}$-root lattice can be realized as the (Hermitian) trace form of the $p$-th cyclotomic extension $\Q(\zeta_{p})/\Q$
restricted to the fractional ideal generated by $(1-\zeta_{p})^{-(p-3)/2}$.
In this paper,
in contrast with the case of the $A_{p-1}$-root lattice,
we prove the following theorem:
Let $n$ be an odd positive integer
and $F/\Q$ be a Galois extension of degree $n$.
Then, there exist no fractional ideals $\Lambda \subset F$
such that the restricted trace form $(\Lambda, \Tr|_{\Lambda \times \Lambda})$ is of type $A_{n}, D_{n}, E_{n}$.
The proof is done by the prime ideal factorization of fractional ideals of $F$
with care of certain 2-adic obstruction.
Additionally, we prove that
every cyclic cubic field contains infinitely many distinct lattices of type $A_{3}$ (i.e., normalized face centered cubic lattices) having normal $\Z$-bases.
The latter fact is in contrast with another fact that
among quadratic fields
only $\Q(\sqrt{\pm3})$ contain lattices of type $A_{2}$.
%and that among them only the lattices $\Z+2^{-1}(1+\sqrt{\pm3})\Z$ have normal $\Z$-bases.
\end{abstract}

%\tableofcontents

%%%Contents

%%%
\section{Introduction}

In modern mathematics,
both of lattices and modular forms appear in many research areas
across pure and applied mathematics.
In particular, there are mutual implications
between the root systems of even unimodular lattices and modular forms.
The theory of modular forms can be applied to classify even unimodular lattices
in low dimensions (up to 24 dimensions) by examining their root systems
(see e.g.\ \cite[Ch.\ 3]{Ebeling}).
On the other hand,
for every odd prime number $p$,
Ebeling [4, Ch. 5] constructed a Hilbert modular theta function over $\Q(\zeta_{p}+\zeta_{p}^{-1})$
by constructing an even unimodular lattice from the fractional ideal of $\Q(\zeta_{p})$ generated by $(1-\zeta_{p})^{-(p-3)/2}$ which is a lattice of type $A_{p-1}$.

In order to generalize Ebeling's construction to more general even root lattices, especially irreducible ones of type $A_{n}, D_{n}, E_{n}$,
it is natural to ask which lattices can be realized as ideal lattices over number fields.
For example, the only quadratic fields containing lattice of type $A_{2}$ are $\Q(\sqrt{\pm3})$ (see \S5).
Moreover,
in \cite{BF_Hirzebruch70},
the lattice of type $D_{4}$, $E_{6}$, and $E_{8}$ are constructed as fractional ideals of $\Q(\zeta_{8})$, $\Q(\zeta_{9})$, and $\Q(\zeta_{20})$ etc respectively.
However, as far as the authors know,
there is no literature which determines completely
which lattices of type $A_{n}, D_{n}, E_{n}$ can be realized as fractional ideals of number fields.
%For cyclotomic fields, see e.g.\ \cite[Theorem 3.2]{BF_Hirzebruch70}.
This leads us to the following question:

\begin{question} \label{motivation}
Let $n \geq 2$ be an integer.
\begin{enumerate}
\item
Can one find a totally real or CM number field $F$ of degree $n$
and a fractional ideal $\Lambda$ of $F$ such that
the lattice $(\Lambda, \Tr|_{\Lambda \times \Lambda})$ is of type $A_{n}, D_{n}, E_{n}$?

\item
If the answer to {\rm (1)} is negative for a given $F$,
can one find a sub $\Z$-module $\Lambda$ of $F$ such that
the lattice $(\Lambda, \Tr|_{\Lambda \times \Lambda})$ is of type $A_{n}, D_{n}, E_{n}$?
\end{enumerate}
\end{question}

For every Galois extension of odd degree,
the following \cref{main_ADE} gives a negative answer to \cref{motivation} (1).
This is the first main result of this paper.

%%%
\begin{theorem} \label{main_ADE}
Let $n$ be an odd positive integer
and $F/\Q$ be a Galois extension of degree $n$.
Then, there exist no fractional ideals $\Lambda \subset F$
such that $(\Lambda, \Tr|_{\Lambda \times \Lambda})$ are of type $A_{n}, D_{n}, E_{n}$.
\end{theorem}

\begin{remark}
In the literature,
some results similar to \cref{main_ADE} are stated.
For example,
in \cite[Proposition 4.7]{Jorge-Costa},
a similar result for $D_{n}$ (with additional twisting parameters) is stated.
%under the condition that 2 is unramified in $F/\Q$.
However, since the proof of the key lemma \cite[Lemma 4.6]{Jorge-Costa} depends on
an implicit condition that the ring of algebraic integers in $F$ is monogenic
(see e.g.\ \cite{Kashio-Sekigawa} and the references therein),
the proof of \cite[Proposition 4.7]{Jorge-Costa} does not work
at least for general $F$.
Similar errors occur reproductively in
e.g.\ \cite[Propositions 4.7 and 4.10]{JdACS} and \cite[Propositions 3.7 and 4.9]{de Araujo-Jorge}.
\end{remark}

On \cref{motivation} (2),
there are several known results.
For example,
for any integer $n \geq 3$ and any prime number $p$ such that $p \equiv 1 \bmod{n}$,
de Araujo and Jorge \cite{de Araujo-Jorge}
constructed sub $\Z$-modules of type $D_{n}$ in the unique subfield of $\Q(\zeta_{p})$ of degree $n$.
The second main result of this paper gives a novel construction for $D_{3} = A_{3}$,
which generalizes and refines the above result \cite{de Araujo-Jorge} for $n = 3$ as follows.

%%%
\begin{theorem} \label{main_cubic}
Let $F/\Q$ be a cyclic cubic extension.
Then, there exist infinitely many distinct lattices $\Lambda \subset F$ of type $A_{3}$
which have normal $\Z$-bases.
\end{theorem}

It may be well-known for experts that
every cyclic cubic field $F$ contains such a lattice (see \cref{self-dual}).
It is a novelty of our construction, however, that
we give an \emph{infinite family} of lattices of $F$ of type $A_{3}$ which have normal $\Z$-bases
from a Shanks polynomial $f_{t}(x)$ (cf.\ \cite{Shanks}) whose minimal splitting field is (isomorphic to) $F$.
Moreover, such normal $\Z$-bases are parametrized by rational points on a certain conic explicitly by means of the parameter $t$ of the Shanks polynomial $f_{t}(x)$;
see the proof of \cref{main_cubic}.

The organization of this paper is as follows.
In \S2, we recall some basic definitions.
In \S3, we give a proof of \cref{main_ADE} in a more generalized form (\cref{main_Galois}).
After that,  we prove that the $A_{3}$-root lattice cannot be realized as a fractional ideal of any totally real number fields (see \cref{no_A3}).
In addition,
we give partial results on the non-existence of lattices of type $A_{n}, D_{n}, E_{n}$ in number fields, which are not necessarily fractional ideals.
In \S4, we give a proof of \cref{main_cubic} which is (essentially) constructive.
In \S5, we give a complementary result for quadratic fields mentioned above.
In \S6, we summarize the known facts on lattices of type $A_{n}$ for $n \leq 8$ which are contained in number fields of degree $n$.
In appendices,
we explain supplementary facts on cyclic cubic fields,
specifically the genericity of the Shanks polynomial
and some properties of ideal lattices sharing the discriminant group with the $A_{3}$-root lattice.

%%%
\subsection*{Acknowledgement}
The authors are sincerely grateful Prof.\ Hiroki Aoki for encouraging their research.
They also sincerely appreciate valuable discussions on cyclic cubic fields with Dr.\ Ryutaro Sekigawa.
They would like to thank 
Prof.\ Eva Bayer-Fl\"uckiger, Prof.\ Yuya Matsumoto, and Mr.\ Takato Suzuki
for their careful reading of the early manuscript and giving many valuable comments.
In particular,
they sincerely grateful
Prof.\ Bayer-Fl\"uckiger for pointing out that
a part of our results had already been obtained in her article \cite{BF-Lenstra} 
in more general and precise form.

%%%
\subsection*{Notation}
In this paper,
$\Z, \Z_{p}, \Q, \R, \C$ denote the ring of integers,
the ring of $p$-adic integers with a prime number $p$,
the field of rational numbers,
the field of real numbers,
and the field of complex numbers,
respectively.
We denote the complex conjugate on $\C$ by $\bar{\cdot} : \C \to \C$.
For every integer $n \neq 0$,
set $\zeta_{n} := \exp(2\pi\sqrt{-1}/n) \in \C$.

A subfield $F$ of $\C$ is called a number field if its degree $[F : \Q]$ over $\Q$ is finite.
For every number field $F$,
$\Z_{F}$ denotes the ring of integers in $F$,
that is,
the subring of $F$ consisting of the roots of the monic polynomials with coefficients in $\Z$.
A finitely generated $\Z_{F}$-submodule of $\Z_{F}$ (resp.\ of $F$) is called 
an ideal of $\Z_{F}$ (resp.\ a fractional ideal of $F$).
Since $\Z_{F}$ is a Dedekind domain,
the non-zero ideals of $\Z_{F}$ form a free abelian group,
and hence every non-zero fractional ideal $\Lambda$ of $F$ can be decomposed into
a product $\Lambda = \Lambda_{+}\Lambda_{-}^{-1}$ uniquely with coprime ideals $\Lambda_{+}, \Lambda_{-}$ of $\Z_{F}$
(see e.g.\ \cite{Samuel}).

%For every prime number $p$,
%$v_{p} : \Q^{\times} \to \Z$ denotes the $p$-adic valuation map,
%that is,
%the unique group homomorphism satisfying $v_{p}(p) = 1$ and $v_{p}(l) = 0$ for every prime number $l \neq p$.
%For the usual convention, $\Z_{p}$ denotes the ring of $p$-adic integers.

%%%
\section{Basic definitions}

First, we recall some basic definitions.
For details, see e.g.\ \cite[Ch.\ 1]{Ebeling}.

\begin{definition}
Let $\Lambda$ be a free $\Z$-module of rank $n$
and $\langle\ ,\ \rangle : \Lambda \times \Lambda \to \R$ be a bilinear form.
\begin{enumerate}
\item
A pair $\Lambda = (\Lambda, \langle\ ,\ \rangle)$ is called a lattice
if $\langle\ ,\ \rangle$ is positive-definite and symmetric,
that is,
$\langle x, x \rangle > 0$ for every $x \in \Lambda \setminus \{ 0 \}$
and $\langle x, y \rangle = \langle y, x \rangle$ for every $x, y \in \Lambda$.

\item
A lattice $\Lambda$ is called integral 
if $\langle x, y \rangle \in \Z$ for every $x, y \in \Lambda$.

\item
An integral lattice $\Lambda$ is called even
if $\langle x, x \rangle \in 2\Z$ for every $x \in \Lambda$,
and odd otherwise.

%\item
%An even lattice $\Lambda$ is called a root lattice if it is generated by the vectors $x \in \Lambda$ such that $\langle x, x \rangle = 2$.

\item
A lattice $\Lambda$ is called of type $A_{n}$
if $\Lambda$ has a basis $(e_{1}, \dots, e_{n})$ such that
\[
	\langle e_{i}, e_{j} \rangle
	= \begin{cases}
	2 & \text{if $\lvert j - i \rvert = 0$, i.e., $j = i$}, \\
	-1 & \text{if $\lvert j - i \rvert = 1$}, \\
	0 & \text{otherwise}.
	\end{cases}
\]

\item
A lattice $\Lambda$ is called of type $D_{n}$ ($n \geq 4$)
if $\Lambda$ has a basis $(e_{1}, \dots, e_{n})$ such that
\[
	\langle e_{i}, e_{j} \rangle
	= \begin{cases}
	2 & \text{if $\lvert j - i \rvert = 0$, i.e., $j = i$}, \\
	-1 & \text{if ($\lvert j - i \rvert = 1$ and $\{ i, j \} \neq \{ 1, 2 \}$) or $\{ i, j \} = \{ 1, 3 \}$}, \\
	0 & \text{otherwise}.
	\end{cases}
\]

\item
A lattice $\Lambda$ is called of type $E_{n}$ ($n = 6, 7, 8$)
if $\Lambda$ has a basis $(e_{1}, \dots, e_{n})$ such that
\[
	\langle e_{i}, e_{j} \rangle
	= \begin{cases}
	2 & \text{if $\lvert j - i \rvert = 0$, i.e., $j = i$}, \\
	-1 & \text{if ($\lvert j - i \rvert = 1$ and $\{ i, j \} \neq \{ 1, 2 \}$) or $\{ i, j \} = \{ 1, 4 \}$}, \\
	0 & \text{otherwise}.
	\end{cases}
\]

\item
For a lattice $\Lambda$ of rank $n$,
its dual lattice $\Lambda^{*}$ is defined by
\[
	\Lambda^{*} := \left\{ x \in \Lambda \otimes_{\Z} \R \setmid \langle x, y \rangle \in \Z \ \text{for every} \ y \in \Lambda \right\},
\]
where we extend $\langle\ ,\ \rangle$ $\R$-bilinearly to $(\Lambda \otimes_{\Z} \R) \times (\Lambda \otimes_{\Z} \R)$.
\end{enumerate}
\end{definition}

Note that a lattice $\Lambda$ is integral if and only if $\Lambda \subset \Lambda^{*}$.
We are interested in integral lattices that arise as sub $\Z$-modules of certain number fields.

\begin{definition}
Let $F$ be a number field.
\begin{enumerate}
\item
$F$ is called totally real if
every field homomorphism $F \to \C$ has the image in $\R$.

\item
$F$ is called CM (complex multiplication)  if $F$ has no field homomorphisms $F \to \R$
and $F$ has a totally real subfield $F^{+}$ such that $[F : F^{+}] = 2$.
\end{enumerate}
\end{definition}

If $F$ is a CM number field,
then $F/F^{+}$ is a Galois extension whose Galois group $\Gal(F/F^{+})$ is generated by the complex conjugate.

\begin{lemma}
Suppose that $F$ is a totally real or CM number field.
Then, the bilinear form
\[
	\Tr = \Tr_{F} : F \times F \to \Q; (x, y) \mapsto \Tr_{F/\Q}(x\bar{y})
\]
is positive-definite and symmetric.
In particular,
for every sub $\Z$-module $\Lambda$ of $F$,
the pair $(\Lambda, \Tr|_{\Lambda \times \Lambda})$ is a lattice.
\end{lemma}

\begin{proof}
The only non-trivial statement is the symmetry of $\Tr$ for a CM number field $F$,
which we can check as follows:
\[
	\Tr_{F/\Q}(y\bar{x})
	= \Tr_{F^{+}/\Q}(\Tr_{F/F^{+}}(y\bar{x}))
	= \Tr_{F^{+}/\Q}(y\bar{x}+x\bar{y})
	= \Tr_{F^{+}/\Q}(\Tr_{F/F^{+}}(x\bar{y}))
	= \Tr_{F/\Q}(x\bar{y}).
\]
\end{proof}

In what follows,
we abbreviate $\Tr|_{\Lambda \times \Lambda}$ to $\Tr$.

The following fact is classically known and a starting point of our study.

\begin{theorem} [{cf.\ \cite[p.\ 76]{BF_Hirzebruch70} or \cite[Lemma 5.4]{Ebeling}}] \label{cyclotomic}
Let $p$ be an odd prime number
and $\Lambda$ be the fractional ideal of $\Q(\zeta_{p})$ generated by $(1-\zeta_{p})^{-(p-3)/2}$.
Then, the lattice $(\Lambda, \Tr)$ is of type $A_{p-1}$.
\end{theorem}

This fact leads us naturally to \cref{motivation} as mentioned in introduction.
In the next section,
we give a proof of \cref{main_ADE},
which gives a negative answer to \cref{motivation} (1)
for every Galois extension of odd degree.

%%%
\section{Non-existence of lattices of type $A_{n}, D_{n}, E_{n}$}
\subsection{Different ideals}

Before the proof of \cref{main_ADE},
we recall some basic facts from \cite[Ch.\ III]{Serre_LF} in our setting.

Let $F$ be a number field.
Then, its different ideal $\mathfrak{D}_{F} = \mathfrak{D}_{\Z_{F}/\Z}$ is defined by the equality
\[
	\mathfrak{D}_{F}^{-1}
	= (\Z_{F}, \Tr)^{*}
	:= \left\{ x \in F \setmid \Tr(xy) \in \Z \ \text{for every} \ y \in \Z_{F} \right\}.
\]
Since the lattice $(\Z_{F}, \Tr)$ is integral,
we have $\Z_{F} \subset \mathfrak{D}_{F}^{-1}$.
Hence, $\mathfrak{D}_{F} \subset \Z_{F}$.
The different ideal is pivotal in the calculation of the dual lattices of fractional ideals.

%%%
\begin{lemma} [{cf.\ \cite[Proposition 7, Ch.\ III, \S3]{Serre_LF}}] \label{dual_computation}
Let $F$ be a number field and $\Lambda$ be a fractional ideal of $F$.
Then, the following conditions are equivalent.
\begin{enumerate}
\item
$\Tr(\Lambda) \subset \Z$.

\item
$\Lambda \subset \mathfrak{D}_{F}^{-1}$,
i.e., $\Tr(\Lambda y) \subset \Z$ for every $y \in \Z_{F}$.
\end{enumerate}
In particular,
if $\Lambda$ is non-zero,
then the following equality holds.
\[
	\Lambda^{*} = \Lambda^{-1}\mathfrak{D}_{F}^{-1}.
\]
\end{lemma}

\begin{proof}
The implication (1)$\Rightarrow$(2) follows from the assumption that $\Lambda$ is a fractional ideal of $F$,
especially $\Lambda y \subset \Lambda$ for every $y \in \Z_{F}$.
The other implication (2)$\Rightarrow$(1) follows from the case $y = 1$.
For the rest of the assertion,
note that $\Lambda^{*}$ is characterized as the maximum among fractional ideals of $F$ such that $\Tr(\Lambda\Lambda^{*}) \subset \Z$,
i.e., $\Lambda\Lambda^{*} \subset \mathfrak{D}_{F}^{-1}$.
This means that $\Lambda^{*} = \Lambda^{-1}\mathfrak{D}_{F}^{-1}$ as claimed.
\end{proof}

The different ideal is also important to describe the factorization of prime numbers in number fields.
One of the most primitive applications is the following criterion for ramification due to Dedekind.

\begin{lemma} [{cf.\ \cite[Theorem 1, Ch.\ III]{Serre_LF}}] \label{Dedekind}
Let $F$ be a number field and $p$ be a prime number.
Then, $p$ is ramified in $F/\Q$ if and only if $p \mid N(\mathfrak{D}_{F})$,
where $N(\mathfrak{D}_{F}) := \#(\Z_{F}/\mathfrak{D}_{F})$.
\end{lemma}

The next lemma is well-known and a key ingredient of the proof of \cref{main_ADE}.
We can understand it as a certain global obstruction for ideal lattices to be of type $A_{n}, D_{n}, E_{n}$.

\begin{lemma} \label{square_discriminant}
Let $F/\Q$ be a Galois extension of odd degree.
Then, $\mathfrak{D}_{F}$ is the square of an ideal of $\Z_{F}$.
\end{lemma}

\begin{proof}
Let $p$ be a prime number ramified in $F/\Q$.
Then, since we assume that $F/\Q$ is Galois,
the ramified index $e_{p}$ of $p$ and the residue degree $f_{p}$ of $p$ are well-defined,
and they satisfy $e_{p}f_{p}g_{p} = [F : \Q]$, where $g_{p}$ denotes the number of prime ideals of $\Z_{F}$ above $p$.
%In particular, $f_{p}$ divides $[F : \Q]$.
Moreover, the ideal $p\Z_{F}$ decomposes into $\prod_{\mathfrak{p} \mid p} \mathfrak{p}^{e_{p}f_{p}}$,
where $\mathfrak{p}$ runs over the prime ideals of $\Z_{F}$ above $p$.
Thus, 
if $\mathfrak{D}_{F} =  \prod_{p \mid N(\mathfrak{D}_{F})} \prod_{\mathfrak{p} \mid p} \mathfrak{p}^{a_{p}}$
denotes the prime ideal decomposition of $\mathfrak{D}_{F}$,
then we have
\[
	N(\mathfrak{D}_{F})\Z_{F}
	= \left( \prod_{p \mid N(\mathfrak{D}_{F})} \prod_{\mathfrak{p} \mid p} (\mathfrak{p} \cap \Z)^{a_{p}f_{p}} \right)\Z_{F}
	%= \prod_{p \mid N(\mathfrak{D}_{F})} \prod_{\mathfrak{p} \mid p} \mathfrak{p}^{a_{p}e_{p}f_{p}^{2}g_{p}}
	= \prod_{p \mid N(\mathfrak{D}_{F})} \prod_{\mathfrak{p} \mid p} \mathfrak{p}^{a_{p}f_{p}[F : \Q]},
\]
%where $g_{p}$ denotes the number of prime ideals of $\Z_{F}$ above $p$
(cf.\ \cite[\S5, Ch.\ I]{Serre_LF}).
On the other hand,
since $N(\mathfrak{D}_{F})\Z_{F}$ is the square of an ideal of $\Z_{F}$
(cf.\ \cite[Remark, p.\ 49]{Serre_LF} and \cite[Proposition 6, Ch.\ III]{Serre_LF}),
we see that $a_{p}f_{p}[F : \Q]$ is even for every prime number $p$.
Since we assume that $[F : \Q]$ is odd,
we see that $f_{p}[F : \Q]$ is also odd,
and hence $a_{p}$ must be even.
This implies that $\mathfrak{D}_{F}$ is square.
\end{proof}

Finally, we summarize the information of $\Lambda^{*}/\Lambda$ for $\Lambda$ is of type $A_{n}, D_{n}, E_{n}$:
\[
	\Lambda^{*}/\Lambda
	\simeq \begin{cases}
		\Z/(n+1)\Z & \text{if $\Lambda$ is of type $A_{n}$}, \\
		(\Z/2\Z)^{\oplus 2} & \text{if $\Lambda$ is of type $D_{n}$ for even $n$}, \\
		\Z/4\Z & \text{if $\Lambda$ is of type $D_{n}$ for odd $n$}, \\
		\Z/(9-n)\Z & \text{if $\Lambda$ is of type $E_{n}$}.
		\end{cases}
\]
In particular,
$(\Lambda^{*}/\Lambda) \otimes_{\Z} \Z_{2}$ is a non-zero cyclic $\Z_{2}$-module whenever $n$ is odd.

%%%
\subsection{Proof of \cref{main_ADE}}

In view of \cref{square_discriminant},
\cref{main_ADE} is a consequence of the following theorem for $p = 2$:

\begin{theorem} \label{main_Galois}
Let $F/\Q$ be a Galois extension (of degree $> 1$) such that $\mathfrak{D}_{F}$ is the square of an ideal of $\Z_{F}$.
Let $p$ be a prime number
and $\Lambda$ be a non-zero fractional ideal of $F$
such that $(\Lambda, \Tr)$ is integral and $(\Lambda^{*}/\Lambda) \otimes_{\Z} \Z_{p}$ is a non-zero cyclic $\Z_{p}$-module.
Then, there exists some $x \in \Lambda$ such that $\Tr(x^{2}) \not\in p\Z$.
\end{theorem}

\begin{proof}[Proof of \cref{main_Galois}]
Suppose that $(\Lambda, \Tr)$ is integral
and $(\Lambda^{*}/\Lambda) \otimes_{\Z} \Z_{p} \simeq \Z/p^{v}\Z$ as a $\Z_{p}$-module
with some $v > 0$.
%If $v \geq 0$, then $\mathfrak{p}$ is not inert.
Then,
by the (unique) prime ideal factorization,
\footnote{
	We also use implicitly
	an isomorphism $\Lambda/\mathfrak{a}\Lambda \simeq \Z_{F}/\mathfrak{a}$ of $\Z_{F}$-modules
	for every non-zero ideal $\mathfrak{a}$ of $\Z_{F}$,
	which can be verified by localizing $\Z_{F}$.
	}
there exists a prime ideal $\mathfrak{p}$ of $\Z_{F}$ above $p$
%In fact, $\mathfrak{p}$ is not inert in $F/\Q$ even if $[F : \Q]$ is even.
and an ideal $\mathfrak{a}$ of $\Z_{F}$ which is prime to $p$
such that
\[
	\Lambda = \mathfrak{p}^{v}\mathfrak{a}\Lambda^{*}.
\]
Moreover,
\cref{dual_computation} shows that
\[
	\Lambda^{*} = \Lambda^{-1} \mathfrak{D}_{F}^{-1},
	\quad \text{i.e.,} \quad
	\mathfrak{p}^{v}\mathfrak{a} = \Lambda^{2}\mathfrak{D}_{F}.
\]
Since we assume that $\mathfrak{D}_{F}$ is square,
$v$ must be even, say $2w$,
and $\mathfrak{a}\mathfrak{D}_{F}^{-1}$ must be the square of a fractional ideal of $F$,
say $\mathfrak{b}^{2}$,
so that
\[
	\Lambda
	= \mathfrak{p}^{w}\mathfrak{b}
	= \mathfrak{p}^{w}\mathfrak{b}_{+}\mathfrak{b}_{-}^{-1}.
\]
On the other hand,
since we assume that $(\Lambda^{*}/\Lambda) \otimes_{\Z} \Z_{p} \simeq \Z/p^{v}\Z$ with some $v > 0$,
we see that $v = 2w \geq 2$.
Hence, $\mathfrak{p}$ cannot be ramified nor inert in $F/\Q$.
Since $F/\Q$ is a Galois extension,
this means that $p$ must split completely in $F/\Q$,
that is,
\[
	p\Z_{F} = \prod_{\sigma \in \Gal(F/\Q)} \mathfrak{p}^{\sigma}
	\quad \text{and} \quad
	\mathfrak{p}^{\sigma} = \mathfrak{p}
	\quad \text{if and only if} \quad \sigma = \op{id}.
\]
In particular,
\cref{Dedekind} shows that
$\mathfrak{D}_{F}$, hence $\mathfrak{b}_{\pm}$ are prime to $p$.

In what follows,
for the localization $\Z_{F, \mathfrak{p}}$ of $\Z_{F}$ at $\mathfrak{p}$,
we prove that $\Tr(x^{2}) \not\in \mathfrak{p}\Z_{F, \mathfrak{p}} \ (\supset p\Z)$ for some $x \in \Lambda = \mathfrak{p}^{w}\mathfrak{b}$.
Indeed,
for every $\tau \in \Gal(F/\Q) \setminus \{ \op{id} \}$,
the Chinese remainder theorem yields a natural surjective map
\[
	\varphi : \Z_{F} \to \Z_{F}/p^{w}\mathfrak{b}_{+}
	\simeq \Z_{F}/(\mathfrak{p}^{w})^{\tau} \times \prod_{\sigma \neq \tau} \Z_{F}/(\mathfrak{p}^{w} )^{\sigma}\times \Z_{F}/\mathfrak{b}_{+}.
\]
Therefore, we have some $x \in \varphi^{-1}(\{(1, 0, \dots, 0)\})$,
which satisfies $x \in (\prod_{\sigma \neq \op{id}, \tau} \mathfrak{p}^{\sigma}) \Lambda \setminus \mathfrak{p}^{\tau}\mathfrak{b}$.
%$x \in p^{w}(\mathfrak{p}^{-w})^{\tau}\mathfrak{b} \setminus (\mathfrak{p}^{w})^{\tau}\mathfrak{b}$.
Then, we see that $(x^{\sigma})^{2} \not\in \mathfrak{p}\Z_{F, \mathfrak{p}}$ if and only if $\sigma = \tau^{-1}$.
%\[
%	(x^{\sigma})^{2} \in \mathfrak{p}\Z_{F, \mathfrak{p}}
%	\quad \text{for every $\sigma \neq \tau^{-1}$ and} \quad
%	(x^{\tau^{-1}})^{2} \not\in \mathfrak{p}\Z_{F, \mathfrak{p}},
%\]
Hence, $\Tr(x^{2}) = \sum_{\sigma \in \Gal(F/\Q)} (x^{\sigma})^{2} \not\in \mathfrak{p}$ as claimed.
This completes the proof.
\end{proof}

\begin{remark}
In the above proof,
we have also proven that $p$ always splits completely in $F/\Q$
under the setting of \cref{main_Galois}.
It might be valuable to note that
Popov and Zarhin \cite[Appendix B]{Popov-Zarhin} gave a criterion for the cyclicity of $\Z_{F}/\mathfrak{D}_{F}$,
whose proof has a similar flavor to our proof of \cref{main_Galois}.
\end{remark}

%%%
\subsection{Some results on lattices of type $A_{n}, D_{n}, E_{n}$}

Let $F$ be a number field of degree $n$.
In this subsection,
we give some results on sub $\Z$-modules of a number field $F$
which is not necessarily a fractional ideal of $F$.
Let $(e_{i})_{1 \leq i \leq n}$ be a $\Z$-basis of $\Z_{F}$
and $\{ \sigma_{j} : F \to \C \mid 1 \leq j \leq n \}$ be the set of ring homomorphisms.
Then, the discriminant $d_{F}$ of $F$ is defined by
\[
	d_{F}
	= \det((\Tr(e_{i}e_{j}))_{1 \leq i, j \leq n})
	= \det((e_{i}^{\sigma_{j}})_{1 \leq i, j \leq n})^{2},
\]
which is independent of the choice of $(e_{i})_{1 \leq i \leq n}$.
In particular,
if $\Lambda \subset F$ is an integral lattice of full rank,
then $\#(\Lambda^{*}/\Lambda) \equiv d_{F} \bmod{\Q^{\times2}}$.
By combining it with the fact that
\[
	d_{F}
	%= \mathfrak{d}_{F}
	= \#N(\mathfrak{D}_{F})
\]
(see e.g.\ \cite[Ch,\ III, Proposition 6]{Serre_LF}),
we obtain the following result.

\begin{proposition} \label{square_criterion}
Let $F$ be a number field and $d_{F}$ be its discriminant.
\begin{enumerate}
\item
Suppose that $d_{F} \in \Q^{\times 2}$.
%(e.g.\ $F/\Q$ is a Galois extension of odd degree).
Then, there exist no integral lattices $\Lambda \subset F$ such that
$\#(\Lambda^{*}/\Lambda)$ is not a square.
In particular, there exist no lattices $\Lambda \subset F$ of type $A_{n}$ such that $n+1$ is not a square or of type $E_{7}$.

\item
Suppose that $F$ is a totally real number field such that $d_{F} \not\in \Q^{\times 2}$.
Then, there exist no integral lattices $\Lambda \subset F$ such that
$\#(\Lambda^{*}/\Lambda)$ is a square.
In particular,
there exist no lattices $\Lambda \subset F$ of type $A_{l^{2}-1}$, $D_{l+2}$, or $E_{8}$,
where $l \geq 2$ is an integer.
\end{enumerate}
\end{proposition}

\begin{remark}
In \cite{BF-Lenstra},
Bayer-Fl\"uckiger and Lenstra proved that
every Galois extension $L/K$ of odd degree has a self-dual normal $K$-basis,
which gives a refinement of \cref{square_criterion}(1)
when $F/\Q$ is a Galois extension of odd degree.
%(cf.\ \cite{Conner-Perlis} for $K = \Q$)
\end{remark}

By applying \cref{square_criterion}(1) to cyclic extensions of odd prime degree $p$,
we obtain the following result,
which is in contrast with \cref{cyclotomic} for $A_{p-1}$.

\begin{corollary} \label{only_p=3}
Let $p$ be an odd prime number and $F/\Q$ be a cyclic extension of degree $p$.
Then, there exists a lattice $\Lambda \subset F$ of type $A_{p}$ only if $p = 3$.
\end{corollary}

\begin{remark}
As we shall prove in the next section,
the converse of \cref{only_p=3} is also true.
Although this fact has already been known (see e.g.\ \cite{BF-Lenstra,Conner-Perlis}),
we reprove it in a more explicit way.
Our method is not only explicit
but also has another advantage that
we obtain infinitely many distinct lattices of type $A_{3}$ having normal $\Z$-bases
in an arbitrary cyclic cubic field (cf.\ \cref{characterization}).
\end{remark}

Moreover, by combining \cref{main_ADE} for cyclic cubic fields
and \cref{square_criterion}(2) for non-Galois cubic fields,
we obtain the following result.

\begin{corollary} \label{no_A3}
There exist no totally real cubic fields which contain a fractional ideal of type $A_{3}$.
\end{corollary}

%%%
Next,
suppose that $F/\Q$ is a Galois extension with $\Gal(F/\Q) = \{ \sigma_{i} \mid 1 \leq i \leq n \}$.
Then,
since $F$ has a normal $\Q$-basis, say the conjugates of $\alpha$,
we obtain the following relation
\[
	d_{F}
	\equiv \prod_{1 \leq i < j \leq n} (\alpha^{\sigma_{j}} - \alpha^{\sigma_{i}})^{2} \bmod{\Q^{\times 2}}.
\]
Therefore, the condition $d_{F} \not\in \Q^{\times2}$ is equivalent to that
$\Gal(F/\Q)$ acts on the conjugates of $\alpha$ alternatingly.
For example, this is the case if $F/\Q$ is a cyclic extension of even degree.
Therefore, by applying \cref{square_criterion}(2) to $l = 2m+1$,
we obtain the following result.

\begin{corollary} \label{even_cyclic}
Let $F$ be a totally real cyclic field of degree $4m(m+1)$ for some $m \geq 1$.
Then, there exist no lattices $\Lambda \subset F$ of type $A_{4m(m+1)}$ or $E_{8}$.
\end{corollary}

%%%
\section{Proof of \cref{main_cubic}}

In this section,
we prove the following

\begin{theorem}[= \cref{main_cubic}]
Let $F/\Q$ be a cyclic cubic extension.
Then, there exist infinitely many distinct lattices $\Lambda \subset F$ of type $A_{3}$
which have normal $\Z$-bases.
\end{theorem}

In view of \cref{square_criterion}(2),
we obtain the following consequence.

\begin{corollary} \label{characterization}
Let $F/\Q$ be a totally real cubic field.
Then, the following conditions are equivalent to each other:
\begin{enumerate}
\item
$F/\Q$ is Galois.

\item
$F$ contains a lattices of type $A_{3}$.

\item
$F/\Q$ is Galois and $F$ contains infinitely many lattices of type $A_{3}$ having normal $\Z$-bases.
\end{enumerate}
\end{corollary}

\begin{example}[cf.\ \cite{Higa}] \label{Higa_example}
Let $F$ be the minimal splitting field of $f_{0}(x) = x^{3}-3x+1$ over $\Q$.
Then, $F$ is a cyclic cubic field.
Take a root $\epsilon$ of $f_{0}$ and a generator $\sigma$ of the Galois group $\Gal(F/\Q)$,
and set
\[
	\alpha_{i} = 1-2\epsilon^{\sigma^{i}} %(i = 0, 1, 2)
	\quad \text{and} \quad
	\beta_{i} = \frac{1}{6} \left( \alpha_{i} + \sum_{j = 0}^{2} \alpha_{j} \right)
	%= \frac{1}{3}\left( 2 - \epsilon^{\sigma^{i}} - \sum_{j = 0}^{2} \epsilon^{\sigma^{j}}\right)
	= \frac{1}{3}(2-\epsilon^{\sigma^{i}})
\]
for every $i \in \{ 0, 1, 2 \}$.
Then,
\[
	\Lambda
	:= \Z\beta_{0} + \Z\beta_{1} + \Z\beta_{2} \subset F
\]
is a lattice having an obvious normal $\Z$-basis $\beta = (\beta_{0}, \beta_{1}, \beta_{2})$.
Since
\[
	\Tr(1) = 3,
	\quad
	\Tr(\epsilon) = 0,
	\quad
	\Tr(\epsilon\epsilon^{\sigma}) = -3
	\quad \text{and} \quad
	\Tr(\epsilon^{2}) = \Tr(\epsilon)^{2}-2\Tr(\epsilon\epsilon^{\sigma}) = 6,
\]
we see that
\[
	\Tr(\beta_{i}, \beta_{j})
	= \begin{cases}
		\frac{1}{9}\Tr((2-\epsilon^{\sigma^{i}})^{2}) = \frac{1}{9}( 4 \cdot 3 - 4 \cdot 0 + 1 \cdot 6 ) = 2 & (j = i), \\
		 \frac{1}{9}\Tr((2-\epsilon^{\sigma^{i}})(2-\epsilon^{\sigma^{j}})) 
		= \frac{1}{9}( 4 \cdot 3 - 4 \cdot 0 + 1 \cdot (-3) ) = 1 & (j \neq i).
		\end{cases}
\]
Moreover,
if we define another $\Z$-basis $\beta' = (\beta'_{0}, \beta'_{1}, \beta'_{2})$ of $\Lambda$ by
\[
	\left( \begin{matrix}
		\beta'_{0} \\
		\beta'_{1} \\
		\beta'_{2}
		\end{matrix} \right)
	= \left( \begin{matrix}
		1 & 0 & 0 \\
		-1 & 1 & 0 \\
		0 & -1 & 1 
		\end{matrix} \right)
	\left( \begin{matrix}
		\beta_{0} \\
		\beta_{1} \\
		\beta_{2}
		\end{matrix} \right),
\]
then we have
\begin{align*}
	\Tr(\beta'_{0}, \beta'_{0}) 
	&= \frac{1}{9}\Tr((2-\epsilon)^{2}) 
		= \frac{1}{9}( 4 \cdot 3 - 4 \cdot 0 + 1 \cdot 6 ) = 2, \\
	\Tr(\beta'_{0}, \beta'_{1}) 
	&= \frac{1}{9}\Tr((2-\epsilon)(\epsilon-\epsilon^{\sigma})) 
		= \frac{1}{9}( 1 \cdot (-3) + (-1) \cdot 6 ) = -1, \\
	\Tr(\beta'_{0}, \beta'_{2}) 
	&= \frac{1}{9}\Tr((2-\epsilon)(\epsilon^{\sigma}-\epsilon^{\sigma^{2}})) 
		= 0, \\
	\Tr(\beta'_{1}, \beta'_{1}) 
	&= \frac{1}{9}\Tr((\epsilon-\epsilon^{\sigma})^{2})
		= \frac{1}{9}((-2) \cdot (-3) + 2 \cdot 6) = 2, \\
	\Tr(\beta'_{1}, \beta'_{2})
	&= \frac{1}{9}\Tr((\epsilon-\epsilon^{\sigma})(\epsilon^{\sigma}-\epsilon^{\sigma^{2}}))
		= \frac{1}{9}(1 \cdot (-3) + (-1) \cdot 6) = -1, \\
	\Tr(\beta'_{2}, \beta'_{2})
	&= \frac{1}{9}\Tr((\epsilon^{\sigma}-\epsilon^{\sigma^{2}})^{2})
		= \frac{1}{9}((-2) \cdot (-3) + 2 \cdot 6) = 2,
\end{align*}
which shows that $(\Lambda, \Tr)$ is of type $A_{3}$.
\end{example}

In order to prove \cref{main_cubic},
first, recall that
every cyclic cubic extension $F/\Q$ is obtained as the minimal splitting field of a Shanks polynomial \cite{Shanks}
\[
	f_{t}(x) := x^{3} - tx^{2} - (t+3)x -1
\]
for some $t \in \Q \setminus \{ -3/2 \}$.
\footnote{
	Note that $f_{-3/2}(x) = (x+2)(x+1/2)(x-1)$.
	}
For the proof, see appendix.
The above example corresponds to the case of $t = 0$.
In what follows,
for every cyclic cubic field $F$,
we fix $t \in \Q$ so that $F$ is the minimal splitting field of $f_{t}$,
and take a root $\epsilon$ of $f_{t}$ and a generator $\sigma$ of $\Gal(F/\Q)$ so that
\[
	\epsilon^{\sigma}
	= -\frac{1}{1+\epsilon}
	\quad \text{and} \quad
	\epsilon^{\sigma^{2}}
	= -\frac{1}{1+\epsilon^{\sigma}}
	= -\frac{1+\epsilon}{\epsilon}.
\]

\begin{lemma} \label{Q-basis}
In the above setting, the triple $(\epsilon, \epsilon^{\sigma}, \epsilon^{\sigma^{2}})$ forms a normal $\Q$-basis of $F$ if and only if $t \neq 0$.
\end{lemma}

\begin{proof}
It is sufficient to show that
if $t \neq 0$, then $\epsilon, \epsilon^{\sigma}, \epsilon^{\sigma^{2}}$ are $\Q$-linearly independent.
Note that, for every $a, b, c \in \Q$,
the following identity holds:
\begin{align*}
	a\epsilon + b\epsilon^{\sigma} + c\epsilon^{\sigma^{2}}
	&= \frac{1}{\epsilon(1+\epsilon)} \left( a\epsilon^{2}(1+\epsilon) - b\epsilon - c(1+\epsilon)^{2} \right) \\
	%&= \frac{1}{\epsilon(1+\epsilon)} \left( a((t+1)\epsilon^{2} + (t+3)\epsilon + 1) - b\epsilon - c(\epsilon^{2}+2\epsilon+1) \right) \\
	&= \frac{1}{\epsilon(1+\epsilon)} \left( (ta+(a-c))\epsilon^{2} + (((t+1)a-b+2(a-c))\epsilon + (a-c) \right).
\end{align*}
Since we assume that
$f_{t}(x)$ is irreducible in $\Q[x]$
%i.e., $1, \epsilon, \epsilon^{2}$ are $\Q$-linearly independent
and $t \neq 0$,
the above rational function of $\epsilon$ vanishes if and only if $a = b = c = 0$.
%This means that $\epsilon, \epsilon^{\sigma}, \epsilon^{\sigma^{2}}$ are $\Q$-linearly independent as claimed.
\end{proof}

In what follows,
we assume that $t \neq 0$.
It is not essential because $f_{0}(x) = -x^{3}f_{-3}(x^{-1})$.
For every $\lambda = (\lambda_{0}, \lambda_{1}, \lambda_{2}) \in \Q^{\oplus 3}$ and $i \in \Z$,
set
\[
	\langle \lambda, \epsilon^{\sigma^{i}} \rangle
	:= \sum_{j = 0}^{2} \lambda_{j}\epsilon^{\sigma^{i+j}},
	\quad \text{i.e.,} \quad
	\left( \begin{matrix}
		\langle \lambda, \epsilon \rangle \\
		\langle \lambda, \epsilon^{\sigma} \rangle \\
		\langle \lambda, \epsilon^{\sigma^{2}} \rangle
		\end{matrix} \right)
	= \left( \begin{matrix}
		\lambda_{0} & \lambda_{1} & \lambda_{2} \\
		\lambda_{2} & \lambda_{0} & \lambda_{1} \\
		\lambda_{1} & \lambda_{2} & \lambda_{0} \\
		\end{matrix} \right)
	\left( \begin{matrix}
		\epsilon \\
		\epsilon^{\sigma} \\
		\epsilon^{\sigma^{2}}
		\end{matrix} \right).
\]
Let $\Lambda \subset F$ be a lattice having a normal $\Z$-basis.
Then,
\cref{Q-basis} shows that
there exists some $\lambda \in \Q^{\oplus 3}$
such that
\[
	\Lambda
	= \Z\langle \lambda, \epsilon \rangle + \Z\langle \lambda, \epsilon \rangle^{\sigma} + \Z\langle \lambda, \epsilon \rangle^{\sigma^{2}}
	= \Z\langle \lambda, \epsilon \rangle + \Z\langle \lambda, \epsilon^{\sigma}\rangle + \Z\langle \lambda, \epsilon^{\sigma^{2}} \rangle.
\]
In what follows, we fix such $\lambda$ for every lattice $\Lambda$ of $F$.
The following lemma is a consequence of straight forward calculation.

\begin{lemma} \label{base_change}
Let $(\alpha_{0}, \alpha_{1}, \alpha_{2})$ be a $\Z$-basis of $\Lambda$ and $A \in \GL(3, \Z)$.
Then, the following conditions are equivalent to each other:
\begin{enumerate}
\item
The following equality of vectors holds:
\[
	\left(  \begin{matrix}
	 	\alpha_{0} \\
		\alpha_{1} \\
		\alpha_{2}
		\end{matrix} \right)
	= A
	\left(  \begin{matrix}
	 	\langle \lambda, \epsilon \rangle \\
		\langle \lambda, \epsilon \rangle^{\sigma} \\
		\langle \lambda, \epsilon \rangle^{\sigma^{2}}
		\end{matrix} \right).
\]

\item
The following equality of matrices holds:
\[
	(\Tr(\alpha_{i}\alpha_{j}))_{0 \leq i, j \leq 2}
	= A(\Tr(\langle \lambda, \epsilon^{\sigma^{i}} \rangle \cdot \langle \lambda, \epsilon^{\sigma^{j}} \rangle))_{0 \leq i, j \leq 2}{^{t}A}.
\]
\end{enumerate}
\end{lemma}

The following is a key ingredient of our proof of \cref{main_cubic}.
In order to state it,
recall that the discriminant of $f_{t}$ is the square of a non-zero rational number $\delta_{t} := t^{2}+3t+9$.

\begin{lemma} \label{quadratic}
Let $d, e \in \Q$.
%Suppose that $t \neq 0$.
Then, the following conditions are equivalent to each other:
\begin{enumerate}
\item
A cyclic cubic field $F$ contains a lattice $\Lambda$ having a normal $\Z$-basis
$(\langle \lambda, \epsilon \rangle, \langle \lambda, \epsilon \rangle^{\sigma}, \langle \lambda, \epsilon \rangle^{\sigma^{2}})$
satisfying the simultaneous equation
\[
	\begin{cases}
	\Tr(\langle \lambda, \epsilon \rangle^{2}) = d, \\
	\Tr(\langle \lambda, \epsilon \rangle \langle \lambda, \epsilon^{\sigma} \rangle) = e
	\end{cases}
\]

\item
The rational number $d+2e$ is a square in $\Q$
and the following equation has a solution $(x, y) \in \Q^{\oplus 2}$:
\[
	x^{2}+3y^{2} = (d-e)\delta_{t}.
\]
\end{enumerate}
\end{lemma}

\begin{proof}
Define two symmetric homogeneous polynomials in $\lambda = (\lambda_{0}, \lambda_{1}, \lambda_{2}) \in \Q^{\oplus 3}$ as follows:
\[
	L(\lambda) := \sum_{i = 0}^{2} \lambda_{i}
	\quad \text{and} \quad
	Q(\lambda) := \sum_{0 \leq i < j \leq 2} \lambda_{i}\lambda_{j}.
\]
Then, we have
\begin{align*}
	\Tr(\langle \lambda, \epsilon \rangle^{2})
	&= \Tr\left( \sum_{i = 0}^{2} \lambda_{i}\epsilon^{\sigma^{i}} \sum_{j = 0}^{2} \lambda_{j}\epsilon^{\sigma^{j}} \right)
		= \Tr\left( \sum_{k = 0}^{2} \left( \epsilon^{\sigma^{k}} \right)^{2} \lambda_{k}^{2}
			+ \sum_{\substack{0 \leq k, l \leq 2 \\ k \neq l}} \epsilon^{\sigma^{k}}\epsilon^{\sigma^{l}} \lambda_{k}\lambda_{l} \right) \\
	&= \Tr\left( \epsilon^{2} \right) \sum_{k = 0}^{2} \lambda_{k}^{2}
		+ \Tr\left(\epsilon\epsilon^{\sigma} \right) \sum_{\substack{0 \leq k, l \leq 2 \\ k \neq l}} \lambda_{k}\lambda_{l} \\
	%&= \left( \Tr\left( \epsilon \right)^{2} - 2\Tr\left(\epsilon\epsilon^{\sigma} \right) \right) \sum_{k = 0}^{2} \lambda_{k}^{2}
	%	+ 2\Tr\left(\epsilon\epsilon^{\sigma} \right) \sum_{0 \leq k < l \leq 2} \lambda_{k}\lambda_{l} \\
	%&= (t^{2} + 2t+6) \left( L(\lambda)^{2} - 2Q(\lambda) \right) - (2t+6) Q(\lambda) \\
	&= (t^{2} + 2t+6) L(\lambda)^{2} - 2\delta_{t}Q(\lambda)
\end{align*}
and
\begin{align*}
	\Tr(\langle \lambda, \epsilon \rangle \langle \lambda, \epsilon^{\sigma} \rangle)
	&= \Tr\left( \sum_{i = 0}^{2} \lambda_{i}\epsilon^{\sigma^{i}} \sum_{j = 0}^{2} \lambda_{j}\epsilon^{\sigma^{j+1}} \right)
		= \Tr\left( \sum_{k = 0}^{2} \left( \epsilon^{\sigma^{k}}\epsilon^{\sigma^{k+1}} \right) \lambda_{k}^{2}
			+ \sum_{\substack{0 \leq k, l \leq 2 \\ k \neq l}} \epsilon^{\sigma^{k}}\epsilon^{\sigma^{l+1}} \lambda_{k}\lambda_{l} \right) \\
	&= \Tr\left( \epsilon\epsilon^{\sigma} \right) \sum_{k = 0}^{2} \lambda_{k}^{2}
		+ \Tr\left(\epsilon^{2} \right) \sum_{\substack{0 \leq k, l \leq 2 \\ k-l \equiv 1 \bmod{3}}} \lambda_{k}\lambda_{l}
		+ \Tr\left(\epsilon\epsilon^{\sigma} \right) \sum_{\substack{0 \leq k, l \leq 2 \\ k-l \equiv 2 \bmod{3}}} \lambda_{k}\lambda_{l} \\
	&= \Tr\left( \epsilon\epsilon^{\sigma} \right) \sum_{k = 0}^{2} \lambda_{k}^{2}
		+ \left( \Tr\left(\epsilon^{2} \right) + \Tr\left( \epsilon\epsilon^{\sigma} \right) \right) \sum_{0 \leq k < l \leq 2} \lambda_{k}\lambda_{l} \\
	%&= \Tr\left( \epsilon\epsilon^{\sigma} \right) \sum_{k = 0}^{2} \lambda_{k}^{2}
	%	+ \left( \Tr\left(\epsilon \right)^{2} - \Tr\left( \epsilon\epsilon^{\sigma} \right) \right) \sum_{0 \leq k < l \leq 2} \lambda_{k}\lambda_{l} \\
	%&= -(t+3) \left( L(\lambda)^{2} - 2Q(\lambda) \right) + (t^{2}+t+3) Q(\lambda) \\
	&= -(t+3)L(\lambda)^{2} + \delta_{t}Q(\lambda)
\end{align*}
Therefore,
by taking into account of the assumption $t \neq 0$,
the simultaneous equation
\[
	\begin{cases}
	\Tr(\langle \lambda, \epsilon \rangle^{2}) = d, \\
	\Tr(\langle \lambda, \epsilon \rangle \langle \lambda, \epsilon^{\sigma} \rangle) = e
	\end{cases}
\]
is equivalent to the following simultaneous equation:
\[
	\begin{cases}
	t^{2}L(\lambda)^{2} = d+2e, \\
	t^{2}\delta_{t}Q(\lambda) = t^{2}e+(t+3)(d+2e).
	\end{cases}
\]
%Here, note that since $t \in \Q$, we know that $\delta_{t} \neq 0$.
In particular, if $d+2e$ is not a square in $\Q$,
then the above simultaneous equation has no solutions $\lambda \in \Q^{\oplus 3}$.

In what follows,
we assume that $d+2e$ is a square in $\Q$, say $f^{2}$ with $f \in \Q$.
Then, we see that the given simultaneous equation is equivalent to that
\[
	\begin{cases}
	L(t\lambda) = \pm f, \\
	Q(t\lambda) = \frac{t^{2}e + (t+3)f^{2}}{\delta_{t}}.
	\end{cases}
\]
Therefore,
it has a solutions $\lambda \in \Q^{\oplus 3}$ if and only if
there exists some $v \in \Q$ such that
the cubic polynomial in $u$
\[
	g_{v}(u) := u^{3} - fu^{2} + \frac{t^{2}e + (t+3)f^{2}}{\delta_{t}}u + v
\]
splits completely into $(u-t\lambda_{0})(u-t\lambda_{1})(u-t\lambda_{2})$ in $\Q[u]$.
%by changing $\lambda$ to $-\lambda$ if necessary
The latter condition can be restated 
that for some $\lambda_{0} \in \Q$ the cubic polynomial in $u$
\[
	(u-t\lambda_{0})\left( u^{2}-(f-t\lambda_{0})u + \frac{t^{2}e + (t+3)f^{2}}{\delta_{t}} -t\lambda_{0} (f-t\lambda_{0}) \right)
\]
splits completely into $(u-t\lambda_{0})(u-t\lambda_{1})(u-t\lambda_{2})$ with some $\lambda_{1}, \lambda_{2} \in \Q$.
This is the case
if and only if the discriminant of the second quadratic factor
\begin{align*}
	D_{t, e, f}(\lambda_{0})
	&:= (f-t\lambda_{0})^{2} - 4\left( \frac{t^{2}e + (t+3)f^{2}}{\delta_{t}} -t\lambda_{0} (f-t\lambda_{0}) \right) \\
	%&= -3\lambda_{0}^{2} + 2f\lambda_{0} + f^{2} - \frac{4t^{2}e + 4(t+3)f^{2}}{\delta_{t}} \\
	%&= -\frac{1}{3}(3\lambda_{0}-f)^{2} + \frac{4f^{2}}{3} - \frac{4t^{2}e + 4(t+3)f^{2}}{\delta_{t}} \\
	%&= -\frac{1}{3}(3\lambda_{0}-f)^{2} + \frac{4\delta_{t}f^{2} - 12t^{2}e - 4(3t+9)f^{2}}{3\delta_{t}} \\
	&= -\frac{1}{3}(3t\lambda_{0}-f)^{2} + \frac{4t^{2}f^{2} - 12t^{2}e}{3\delta_{t}} \\
	&= \frac{1}{3} \left( \frac{2t}{\delta_{t}} \right)^{2} \left( -\left( \frac{\delta_{t}}{2t} \right)^{2}(3t\lambda_{0} - f)^{2} + (d-e) \delta_{t} \right)
\end{align*}
is a square in $\Q$ for some $\lambda_{0} \in \Q$,
that is,
the following equation has a solution $(x, y) \in \Q^{\oplus 2}$:
\[
	x^{2}+3y^{2} = (d-e)\delta_{t},
\]
where
\[
	(x, y)
	= \left( \frac{\delta_{t}}{2t}(3t\lambda_{0} - f), \frac{\delta_{t}}{2t}\sqrt{D_{t, e, f}(\lambda_{0})} \right).
\]
This completes the proof.
\end{proof}

%%%
\begin{proof}[Proof of \cref{main_cubic}]
Let $\lambda = (\lambda_{0}, \lambda_{1}, \lambda_{2}) \in \Q^{\oplus 3}$
and $\Lambda = \Z\langle \lambda, \epsilon \rangle + \Z\langle \lambda, \epsilon^{\sigma} \rangle + \Z\langle \lambda, \epsilon^{\sigma^{2}} \rangle$.
Then, the denominators of the coefficients of $\epsilon$ for all $v \in \Lambda \subset \Q\epsilon+\Q\epsilon^{\sigma}+\Q\epsilon^{\sigma^{2}}$ are bounded by the product of the denominators of $\lambda_{0}, \lambda_{1}, \lambda_{2}$.
On the other hand,
$\Lambda$ contains a vector $\langle \lambda, \epsilon \rangle$,
for which the coefficient of $\epsilon$ is exactly $\lambda_{0}$.
Therefore,
for every infinite family of $\lambda = (\lambda_{0}, \lambda_{1}, \lambda_{2}) \in \Q^{\oplus 3}$
such that the denominators of $\lambda_{0}$ are unbounded,
the corresponding family of lattices $\Lambda = \Z\langle \lambda, \epsilon \rangle + \Z\langle \lambda, \epsilon^{\sigma} \rangle + \Z\langle \lambda, \epsilon^{\sigma^{2}} \rangle$ includes infinitely many distinct lattices.

In what follows,
we prove that
there exists an infinite family of $\lambda = (\lambda_{0}, \lambda_{1}, \lambda_{2}) \in \Q^{\oplus 3}$
such that the denominators of $\lambda_{0}$ are unbounded
and the corresponding family of lattices $\Lambda = \Z\langle \lambda, \epsilon \rangle + \Z\langle \lambda, \epsilon^{\sigma} \rangle + \Z\langle \lambda, \epsilon^{\sigma^{2}} \rangle$ are of type $A_{3}$.
We may assume that $t \neq 0$ as noted already.
Take a matrix
\[
	A = \left( \begin{matrix}
		1 & 0 & 0 \\
		-1 & 1 & 0 \\
		0 & -1 & 1 
		\end{matrix} \right) \in \GL(3, \Z),
\]
which appeared in \cref{Higa_example}.
Then, as mentioned therein,
it holds that
\[
	A
	\left( \begin{matrix}
		2 & 1 & 1 \\
		1 & 2 & 1 \\
		1 & 1 & 2 
		\end{matrix} \right)
	{^{t}A}
	= \left( \begin{matrix}
		2 & -1 & 0 \\
		-1 & 2 & -1 \\
		0 & -1 & 2 
		\end{matrix} \right).
\]
Therefore,
by \cref{base_change,quadratic},
we can construct lattices $\Lambda = \Z\langle \lambda, \epsilon \rangle + \Z\langle \lambda, \epsilon^{\sigma} \rangle + \Z\langle \lambda, \epsilon^{\sigma^{2}} \rangle$ of type $A_{3}$
from the solutions $(x, y) \in \Q^{\oplus 2}$ of the equation $x^{2}+3y^{2} = \delta_{t}$.

In what follows,
we fix $(d, e, f) = (2, 1, 2)$.
Then, the proof of \cref{quadratic} shows that
every solution $(x, y) \in \Q^{\oplus 2}$ of the equation $x^{2}+3y^{2} = \delta_{t}$
corresponds to $\lambda_{0}$ as follows:
\[
	\lambda_{0} = \frac{2}{3}\left( \frac{x}{\delta_{t}} + \frac{1}{t} \right).
\]
Therefore,
for each fixed $t$,
the denominators of $\lambda_{0}$ are unbounded
if and only if the denominators of $x$ are unbounded.
On the other hand,
there exist only finitely many solutions $(x, y) \in \Q^{\oplus 2}$ of the equation $x^{2}+3y^{2} = \delta_{t}$ with bounded denominators of $x$.
%because there exist only finitely many rational numbers $x$ such that $x^{2} \leq \delta_{t}$ and the denominator of $x$ is bounded.
Finally, note that
the conic defined by $x^{2}+3y^{2} = \delta_{t}$ has infinitely many rational points $(x, y) \in \Q^{\oplus 2}$ parametrized by the slopes $s \in \Q \cup \{ \infty \}$ of the lines passing through a base point $(t+3/2, 3/2)$.
%%%
\begin{comment}
On the other hand,
the solutions $(x, y) \in \Q^{\oplus 2}$ of the equation $x^{2}+3y^{2} = \delta_{t}$
are parametrized by the slopes $s \in \Q \cup \{ \infty \}$ of the lines passing through a base point $(t+3/2, 3/2)$ on the conic defined by $x^{2}+3y^{2} = \delta_{t}$ as follows:
\[
	(x, y)
	= \left( - \frac{(1-3s^{2})\left( t+\frac{3}{2} \right) + 9s}{(1+3s^{2})}, \frac{3}{2}+s\left( x-t-\frac{3}{2} \right) \right).
\]
In particular,
there exist infinitely many solutions $(x, y) \in \Q^{\oplus 2}$ of the equation $x^{2}+3y^{2} = \delta_{t}$.
\end{comment}
%%%
As a conclusion,
we have an (explicitly parameterized) infinite family of $\lambda = (\lambda_{0}, \lambda_{1}, \lambda_{2}) \in \Q^{\oplus 3}$
such that the denominators of $\lambda_{0}$ are unbounded
and the corresponding family of lattices
\[
	\Lambda
	= \Z\langle \lambda, \epsilon \rangle + \Z\langle \lambda, \epsilon^{\sigma} \rangle + \Z\langle \lambda, \epsilon^{\sigma^{2}} \rangle
	\subset F
\]
are of type $A_{3}$.
This completes the proof.
\end{proof}

\begin{remark}
The above proof (and the whole proof of \cref{quadratic}) shows that,
for every fixed cyclic cubic field $F$,
we can construct a sub $\Z$-module $\Lambda$ of $F$
such that $(\Lambda, \Tr)$ is of type $A_{3}$ explicitly.

%%%
\begin{comment}
Since
\[
	\Tr(1) = 3,
	\quad
	\Tr(\epsilon) = t,
	\quad
	\Tr(\epsilon\epsilon^{\sigma}) = -(t+3)
	\quad \text{and} \quad
	\Tr(\epsilon^{2}) = \Tr(\epsilon)^{2}-2\Tr(\epsilon\epsilon^{\sigma}) = t^{2}+2t+6
\]
we have
\begin{align*}
	\det(\Tr(\epsilon^{\sigma^{i}}, \epsilon^{\sigma^{j}})_{1 \leq i, j \leq 3})
	&= \det\left( \begin{matrix}
		t^{2}+2t+6 & -(t+3) & -(t+3) \\
		-(t+3) & t^{2}+2t+6 & -(t+3) \\
		-(t+3) & -(t+3) & t^{2}+2t+6 \\
		\end{matrix} \right) \\
	&= \det\left( \begin{matrix}
		t^{2} & t^{2} & t^{2} \\
		-(t+3) & t^{2}+2t+6 & -(t+3) \\
		-(t+3) & -(t+3) & t^{2}+2t+6 \\
		\end{matrix} \right) \\
	&= t^{2} \det\left( \begin{matrix}
		1 & 0 & 0 \\
		-(t+3) & t^{2}+3t+9 & 0 \\
		-(t+3) & 0 & t^{2}+3t+9 \\
		\end{matrix} \right) \\
	& = t^{2}(t^{2}+3t+9)^{2}.
\end{align*}
This with \cref{Q-basis} shows that $t = 0$ is exceptional in our proof at least technically.
However,
since the minimal splitting field of $f_{0}(x)$ and $f_{-3}(x)$ the same,
the assumption $t \neq 0$ is not essential.
\end{comment}
%%%
\end{remark}

\begin{remark} \label{self-dual}
It is known that
a general Galois extension $L/K$ of odd degree
has a self-dual normal $K$-basis \cite{BF-Lenstra}.
In the case of cyclic cubic extension $F$ of $\Q$,
we can say more:
If we take $(d, e, f) = (1, 0, 1)$ instead of $(2, 1, 2)$ in the proof of \cref{main_cubic},
then we obtain the same equation $x^{2}+3y^{2} = \delta_{t}$ with a different correspondence
\[
	\lambda_{0} = \frac{2}{3}\left( \frac{x}{\delta_{t}} + \frac{1}{2t} \right).
\]
From this correspondence,
we obtain infinitely many distinct lattices $\Lambda \subset F$
which have self-dual normal $\Z$-bases and are parametrized explicitly.
%%%
Note that
it is an immediate consequence of \cite{BF-Lenstra} (or \cite{Conner-Perlis})
that $F$ contains at least one lattice of type $A_{3}$
since the equality
\[
	\left( \begin{matrix}
		2 & 1 & 1 \\
		1 & 2 & 1 \\
		1 & 1 & 2
		\end{matrix} \right)
	= \left( \begin{matrix}
		0 & 1 & 1 \\
		1 & 0 & 1 \\
		1 & 1 & 0
		\end{matrix} \right)
		\left( \begin{matrix}
		1 & 0 & 0 \\
		0 & 1 & 0 \\
		0 & 0 & 1
		\end{matrix} \right)
		\left( \begin{matrix}
		0 & 1 & 1 \\
		1 & 0 & 1 \\
		1 & 1 & 0
		\end{matrix} \right)
\]
gives an explicit transform of a self-dual $\Q$-basis to a $\Z$-basis of a lattice of type $A_{3}$.
\end{remark}

%%%
\section{Quadratic fields containing lattices of type $A_{2}$}

In this section,
we prove a complementary result for lattices of type $A_{2}$ in quadratic fields.
Let $d \geq 1$ be a square-free positive integer
\footnote{
	Even if $d$ is not square-free,
	we can replace it to another $d'$ which is square-free
	such that $\Q(\sqrt{d}) \simeq \Q(\sqrt{d'})$.
	}
and
\[
	F = \Q(\sqrt{\pm d}) \simeq \Q[t]/(t^{2} \mp d);
	\sqrt{\pm d} \mapsto t
\]
be a quadratic \'etale $\Q$-algebra.
Then, the bilinear form
$\Tr : F \times F \to \Q; (\alpha, \beta) \mapsto \Tr(\alpha\bar{\beta})$
is positive-definite and symmetric,
where $\bar{\cdot}$ denotes the complex conjugate.
Here, note that
\[
	\Tr(x+y\sqrt{\pm d}) = 2x
	\quad (x, y \in \Q)
\]
regardless of whether $\pm d = 1$ or not.

\begin{proposition}
In the above setting,
there exists a lattice $\Lambda \subset F$ of type $A_{2}$
if and only if $d = 3$.
Moreover, if $d = 3$,
then there exist infinitely many distinct lattices $\Lambda \subset F$ of type $A_{2}$,
and among them, only $\Z \cdot (1+\sqrt{\pm3})/2+\Z \cdot (1-\sqrt{\pm3})/2$ has the normal $\Z$-basis.
\end{proposition}

\begin{proof}
Since $F$ has a $\Q$-basis $(1, \sqrt{\pm d})$,
it contains a lattice of type $A_{2}$ having a $\Z$-basis
$(x_{1}+y_{1}\sqrt{\pm d}, x_{2}+y_{2}\sqrt{\pm d})$ for $x_{1}, y_{1}, x_{2}, y_{2} \in \Q$
if and only if the following simultaneous equation holds:
\[
	\begin{cases}
	\Tr((x_{1}+y_{1}\sqrt{\pm d})(\overline{x_{1}+y_{1}\sqrt{\pm d}})) = 2, \\
	\Tr((x_{2}+y_{2}\sqrt{\pm d})(\overline{x_{2}+y_{2}\sqrt{\pm d}})) = 2, \\
	\Tr((x_{1}+y_{1}\sqrt{\pm d})(\overline{x_{2}+y_{2}\sqrt{\pm d}})) = -1, \\
	\end{cases}
\]
i.e.,
\[
	\begin{cases}
	x_{1}^{2}+dy_{1}^{2} = 1, \\
	x_{2}^{2}+dy_{2}^{2} = 1 ,\\
	2x_{1}x_{2} + 2dy_{1}y_{2} = -1. \\
	\end{cases}
\]

First, we prove the only if part by contradiction.
Suppose that $d \neq 1, 3$ and
the last simultaneous equation has a solution $(x_{1}, y_{1}, x_{2}, y_{2}) \in \Q^{\oplus 4}$.
Then,
the following simultaneous equation has a solution $(X_{1}, Y_{1}, Z_{1}, X_{2}, Y_{2}, Z_{2}) \in \Z^{\oplus 6}$
such that $\gcd(X_{1}, Y_{1}, Z_{1}) = \gcd(X_{2}, Y_{2}, Z_{2}) = 1$:
\[
	\begin{cases}
	X_{1}^{2}+dY_{1}^{2} = Z_{1}^{2}, \\
	X_{2}^{2}+dY_{2}^{2} = Z_{2}^{2} ,\\
	2X_{1}X_{2} + 2dY_{1}Y_{2} = -Z_{1}Z_{2}. \\
	\end{cases}
\]
Since $X_{1} \equiv \pm Z_{1} \bmod{d}$ and $X_{2} \equiv \pm Z_{2} \bmod{d}$,
we must have $\pm 2Z_{1}Z_{2} \equiv -Z_{1}Z_{2} \bmod{d}$.
Moreover, since we assume that $d \neq 1, 3$,
the last congruence implies that
$d$ has a prime divisor $p$ such that $Z_{1} \equiv 0 \bmod{p}$ or $Z_{2} \equiv 0 \bmod{p}$.
If $Z_{1} \equiv 0 \bmod{p}$,
then we must have $X_{1} \equiv 0 \bmod{p}$.
Moreover, since we assume that $d$ is square-free,
we must have $Y_{1} \equiv 0 \bmod{d}$,
which contradicts that $\gcd(X_{1}, Y_{1}, Z_{1}) = 1$.
The case $Z_{2} \equiv 0 \bmod{p}$ is exactly similar.
This completes the proof of the only if part for $d \neq 1$.
For $d = 1$,
note that every solution of the equation
$x_{i}^{2}+y_{i}^{2} = 1$ is given by $(x_{i}, y_{i}) = ((s_{i}^{2}-1)/(s_{i}^{2}+1), 2s_{i}/(s_{i}^{2}+1))$ with some $s_{i} \in \Q$,
but the third equation
\[
	2x_{1}x_{2} + 2y_{1}y_{2} = -1,
	\quad \text{i.e.}, \quad
	%2(s_{1}^{2}-1)(s_{2}^{2}-1) + 8s_{1}s_{2} + (s_{1}^{2}+1)(s_{2}^{2}+1) = 0
	%3s_{1}^{2}s_{2}^{2} - s_{1}^{2} + 8s_{1}s_{2} - s_{2}^{2} + 3 = 0
	%3(s_{1}s_{2}+1)^{2} = (s_{1}-s_{2})^{2} = 0
	3(s_{1}s_{2}+1)^{2} = (s_{1}-s_{2})^{2}
\]
implies that $\sqrt{3} \in \Q$ or $s_{1} = \pm\sqrt{-1} \in \Q$, which is impossible.

For the if part,
note that the only solutions satisfying $x_{1} = x_{2}$ and $y_{1} = -y_{2}$ are
\[
	(x_{1}, y_{1}, x_{2}, y_{2})
	= \left( \frac{1}{2}, \frac{1}{2}, \frac{1}{2}, -\frac{1}{2} \right), \left( -\frac{1}{2}, -\frac{1}{2}, -\frac{1}{2}, \frac{1}{2} \right)
\]
which yield the (same) lattice $\Z \cdot (1+\sqrt{\pm3})/2+\Z \cdot (1-\sqrt{\pm3})/2 \subset \Q(\sqrt{\pm 3})$.
Hence, this is the unique lattice in $\Q(\sqrt{\pm 3})$ of type $A_{2}$
having a normal $\Z$-basis.

Finally,
for the infinitude of the lattices of type $A_{2}$ in $\Q(\sqrt{\pm 3})$,
we parametrize the solutions of the above simultaneous equation for $d = 3$.
First, 
the solutions $(x_{1}, y_{1}) \in \Q^{\oplus 2}$ of the single equation $x_{1}^{2}+3y_{1}^{2} = 1$
can be parametrized by the slopes $s = s_{1}/s_{0} \in \Q \cup \{ \infty \}$ ($s_{0}, s_{1} \in \Z$) of the lines passing through a base point $(1, 0)$ as follows:
\[
	(x_{1}, y_{1})
	= \left( -\frac{s_{0}^{2}-3s_{1}^{2}}{s_{0}^{2}+3s_{1}^{2}}, -\frac{2s_{0}s_{1}}{s_{0}^{2}+3s_{1}^{2}} \right).
\]
Here, we understand that $s = \infty$ gives the base point $(x_{1}, y_{1}) = (1, 0)$,
and we see that $x_{1} \neq 0$ for every $s \in \Q \cup \{ \infty \}$.
Then, the third equation $2x_{1}x_{2} + 6y_{1}y_{2} = -1$
implies that
\[
	x_{2}
	%= \frac{-6y_{1}y_{2}-1}{2x_{1}}
	= - \frac{12s_{0}s_{1}y_{2} - (s_{0}^{2}+3s_{1}^{2})}{2(s_{0}^{2}-3s_{1}^{2})},
\]
and the second equation $x_{2}^{2}+3y_{2}^{2} = 1$ implies that
%\[
%	(12s_{0}s_{1}y_{2} - (s_{0}^{2}+3s_{1}^{2}))^{2} + 12(s_{0}^{2}-3s_{1}^{2})^{2}y_{2}^{2} = 4(s_{0}^{2}-3s_{1}^{2})^{2},
%\]
%i.e.,
\[
	%12(1+3s^{2})^{2}y_{2}^{2} - 24s(1+3s^{2})y_{2} - (1-9s^{2})(3-3s^{2}) = 0,
	%12(1+3s^{2})^{2}y_{2}^{2} - 24s(1+3s^{2})y_{2} - 3(1+3s)(1-3s)(1+s)(1-s) = 0,
	%(2(1+3s^{2})y_{2} + (1+3s)(1-s))(2(3s^{2}+1)y_{2} - (1-3s)(1+s)) = 0,
	y_{2} = -\frac{(s_{0}+3s_{1})(s_{0}-s_{1})}{2(s_{0}^{2}+3s_{1}^{2})}, \frac{(s_{0}-3s_{1})(s_{0}+s_{1})}{2(s_{0}^{2}+3s_{1}^{2})}.
\]
The above solutions $(x_{1}, y_{1}, x_{2}, y_{2})$ parametrize
the $\Z$-bases $(x_{1}+y_{1}\sqrt{\pm 3}, x_{2}+y_{2}\sqrt{\pm 3})$ of some lattices $\Lambda_{s_{0}, s_{1}} \subset \Q(\sqrt{\pm 3})$ of type $A_{2}$.
On the other hand,
for $i = 1, 2$,
there exist only finitely many solutions $(x_{i}, y_{i}) \in \Q^{\oplus 2}$
of the equation $x_{i}^{2}+3y_{i}^{2} = 1$
with bounded denominators of $x_{i}$.
By comparing the coefficients of $1$ for vectors in $\Lambda_{s_{0}, s_{1}} \subset \Q+\Q\sqrt{\pm 3}$,
we see that the family $(\Lambda_{s_{0}, s_{1}})_{s_{0}, s_{1} \in \Z}$
contain infinitely many distinct lattices
(cf.\ the proof of \cref{main_cubic} in \S4).
This completes the proof.
\end{proof}

%%%
\section{Summary}

The following is a summary for lattices of $A_{n}$ with small $n$:
\begin{enumerate}
\item
$F = \Q$ contains no lattices of type $A_{1}$.
(cf.\ $(\Z, 2\Tr) = (\Z, \Tr(2 \cdot, \cdot))$ is of type $A_{1}$.)

\item
Among quadratic fields,
only $F = \Q(\sqrt{-3})$ and $\Q(\sqrt{3})$ contain lattices of type $A_{2}$,
and they also contains fractional ideals of type $A_{2}$ having normal $\Z$-basis (cf.\ \S5).

\item
No totally real cubic fields contain fractional ideals of type $A_{3}$ (cf.\ \cref{no_A3}),
and a totally real cubic field $F$ contains a lattice of type $A_{3}$ if and only if it is Galois.
If this is the case, $F$ also contains infinitely many lattices of type $A_{3}$ having normal $\Z$-basis (cf.\ \S 4).

\item
$F = \Q(\zeta_{5})$ contains a fractional ideal $(1-\zeta_{5})^{-1}$ of type $A_{4}$ (cf.\ \cref{cyclotomic})

\item
No cyclic number fields of degree $5$ contain lattices of type $A_{5}$ (cf.\ \cref{only_p=3}).

\item
$F = \Q(\zeta_{5})$ contains a fractional ideal $(1-\zeta_{7})^{-2}$ of type $A_{6}$ (cf.\ \cref{cyclotomic})

\item
No cyclic number fields of degree $7$ contain lattices of type $A_{7}$ (cf.\ \cref{only_p=3}).

\item
No cyclic number fields of degree $8$ contain lattices of type $A_{8}$ (cf.\ \cref{even_cyclic}).
\end{enumerate}

Thus, it is natural to ask

\begin{question}
Does there exist a (Galois) number field $F$ of degree $8$
and a fractional ideal (or a lattice) $\Lambda$ of $F$ such that $(\Lambda, \Tr)$ is of type $A_{8}$?
\end{question}

Finally,
note that lattices of $D_{4}$, $E_{6}$, and $E_{8}$ can be obtained as fractional ideals of e.g.\ $F = \Q(\zeta_{8}), \Q(\zeta_{9})$, and $\Q(\zeta_{20})$, respectively.
For details including the cases of the Coxeter-Todd lattice of rank 12 and the Leech lattice of rank 24, see e.g.\ \cite[pp. 76--78]{BF_Hirzebruch70}.
In contrast,
it is known that $D_{2n}$ for $n \geq 3$ (and $A_{n}$ for $n \neq p-1$ for any prime number $p$) cannot be realized in cyclotomic fields \cite[Theorem 3.2]{BF_Hirzebruch70}.
While the authors cannot find references for $D_{2n}$ except for the case of cyclotomic fields,
it is clear that one should take care of prime ideals $\mathfrak{p}$ of $\Z_{F}$ above $2$
so that $\Lambda^{*}/\Lambda \simeq (\Z/2\Z)^{\oplus 2}$.

%%%
\appendix
\section{Genericity of Shanks polynomial}

Let $k$ be a field,
and set $f_{t}(x) = x^{3} - tx^{2} - (t+3)x - 1 \in k[x, t]$.

\begin{proposition} \label{genericity}
Let $k$ be a field and $K/k$ be a cyclic cubic extension.
For simplicity, suppose that $k$ contains no primitive third roots of unity.
\footnote{
	If $k$ contains a primitive root of unity,
	then Kummer theory shows that $K \simeq k[x]/(x^{3}-t)$ for some $t \in k$.
	}
Then, there exists an element $t \in k$ such that
$K \simeq k[x]/f_{t}(x)$,
i.e., $K$ is generated by a (single) root of $f_{t}$ over $k$.
\end{proposition}

\begin{proof}
Let $\sigma$ be a fixed generator of the Galois group of $K/k$.
First we take an element $\alpha \in \op{Ker}(\Tr : K \to k) \setminus k$,
which is possible because the trace map $\Tr$ is $k$-linear and $[K : k] \geq 3$.
Set $u = \alpha^{\sigma-1}$.
Then, we see that
\[
	1+u+u^{1+\sigma}
	= 1+\alpha^{\sigma-1}+\alpha^{\sigma^{2}-1}
	= \alpha^{-1}\Tr(\alpha)
	= 0
\]
and $N(u) = 1$.
This implies that
\begin{align*}
	f_{\Tr(u)}(x)
	&= x^{3} - \Tr(u)x^{2} - \Tr(u+1)x - 1
		= x^{3} - \Tr(u)x^{2} + \Tr(u^{1+\sigma})x - N(u) \\
	&= (x-u)(x-u^{\sigma})(x-u^{\sigma^{2}}).
\end{align*}
On the other hand,
since we assume that $k$ contains no primitive third roots of unity,
the conditions $\alpha \not\in k$ and $N(u) = 1$ imply that $u \not\in k$.
Therefore, $K$ is generated by a single root $u$ of $f_{\Tr(u)}$ as desired.
\end{proof}

The above proof is suggested by \cite[Lemma 4.2]{Kashio-Sekigawa}.

%%%
\section{``Fake'' $A_{3}$-root lattices}

Let $F/\Q$ be a cyclic cubic field
which is the minimal splitting field of a Shanks polynomial $f_{t}$.
Then, we see that
2 is unramified in $F$,
and it splits completely in $F$ if and only if $t \not\in \Z_{2}$.

If 2 is totally inert in $F/\Q$,
then there exist no fractional ideals of discriminant 4
because $\Z_{F}/2\Z_{F} \simeq (\Z/2\Z)^{\oplus 3}$ as $\Z$-modules.

If 2 splits completely in $F/\Q$,
then there exist exactly three fractional ideals $\Lambda$ of discriminant 4
according to the choice of the prime ideal $\mathfrak{p}$ above 2,
that is,
\[
	\Lambda = \mathfrak{p}\mathfrak{C}_{F}^{-1}
	\quad \text{with} \quad
	\mathfrak{C}_{F}^{2} = \mathfrak{D}_{F}.
\]
Note that
the lattice $(\mathfrak{C}_{F}^{-1}, \Tr)$ is integral
\footnote{
	In fact,
	\cref{dual_computation} implies that
	the lattice $(\mathfrak{C}_{F}^{-1}, \Tr)$ is a unimodular lattice of rank $[F : \Q] = 3$.
	In particular, the corresponding ternary quadratic form is $x^{2}+y^{2}+z^{2}$,
	which is independent of $F$.
	}
because
\[
	\Tr(\mathfrak{C}_{F}^{-1} \times \mathfrak{C}_{F}^{-1})
	\subset \Tr\left( \mathfrak{C}_{F}^{-2} \right)
	= \Tr\left( \mathfrak{D}_{F}^{-1} \right)
	\subset \Z.
\]
In particular, a sublattice $(\mathfrak{p}\mathfrak{C}_{F}^{-1}, \Tr)$ of $(F, \Tr)$
is an odd lattice such that $(\mathfrak{p}\mathfrak{C}_{F}^{-1})^{*}/\mathfrak{p}\mathfrak{C}_{F}^{-1} \simeq \Z/4\Z$.
In this sense,
it is a certain ``fake'' $A_{3}$-lattice.
According to the Brandt-Intrau table \cite{LATTICES},
such a lattice is isometric to the one associated with the ternary quadratic form $x^{2}+y^{2}+4z^{2}$.

Finally, note also that
both of $\mathfrak{p}\mathfrak{C}_{F}^{-1}$ and its dual lattice $(\mathfrak{p}\mathfrak{C}_{F}^{-1})^{*} = \mathfrak{p}^{-1}\mathfrak{C}_{F}^{-1}$ are not closed under the action of $\Gal(F/\Q)$,
hence they cannot have normal $\Z$-bases.
This property is also comparable with the lattices of type $A_{3}$ constricted in \S4.

\end{document}